\documentclass[11pt]{article}
\usepackage{amssymb,latexsym,times}
\usepackage[dvips]{color}
\usepackage{mathptmx}
\usepackage{verbatim}
\usepackage{amsfonts}
\usepackage{amsthm}
\usepackage{amscd}
\usepackage[mathscr]{eucal}
\catcode`\@=11
\renewcommand\subsection{\@startsection{subsection}{2}{\z@}%
                                     {-3.25ex\@plus -1ex \@minus -.2ex}%
                                     {-0.01 mm}
                                     {\normalfont\large\bfseries}}

\catcode`\@=11
\renewcommand\subsubsection{\@startsection{subsubsection}{2}{\z@}%
                                     {-3.25ex\@plus -1ex \@minus -.2ex}%
                                     {-0.01 mm}
                                     {\normalfont\bfseries}}
\newtheorem{theorem}{Theorem}[section]
\newtheorem{example}[theorem]{Example}
\newtheorem{definition}[theorem]{Definition}
\newtheorem{proposition}[theorem]{Proposition}
\newtheorem{lemma}[theorem]{Lemma}
\newtheorem{remark}[theorem]{Remark}

\def\cqfd{\hfill $\Box$ \bigskip}
\def\adots{\mathinner{\mkern2mu\raise1pt\hbox{.}
\mkern3mu\raise4pt\hbox{.}\mkern1mu\raise7pt\hbox{.}}}
\def\<{\langle\,}
\def\>{\,\rangle}

\def\ie{{\em i.e.\,}}

\def\a{\alpha}

\def\b{\beta}

\def\N{{\mathbb N}}
\def\Z{{\mathbb Z}}
\def\C{{\mathbb C}}

\def\Q{{\mathbb Q}}
\def\K{\mathcal{K}}

\def\F{{\cal F}}
\def\P{{\cal P}}

\def\hI{\widehat{I}\,}

\def\SS{{\cal S}}

\def\g{\mathfrak g}

\def\A{{\cal A}}

\def\<{\langle}
\def\>{\rangle}

\def\CC{{\cal C}}

\def\bchi{\widetilde{\chi}}

\def\M{{\cal M}}
\def\RR{{\cal R}}

\def\le{\leqslant}
\def\ge{\geqslant}

\def\xx{{\mathbf x}}

\def\zz{{\mathbf z}}
\def\Y{{\mathcal Y}}

\def\isom{\stackrel{\sim}{\rightarrow}}

\title{\bf Monoidal categorifications of cluster algebras \\ of type $A$ and $D$}
\author{David Hernandez and Bernard Leclerc}
\date{}

\begin{document}
\maketitle

\begin{center}
{\it To M. Jimbo on his 60th birthday}
\end{center}

\begin{abstract} In this note, we introduce monoidal subcategories of the tensor category of finite-dimen\-sio\-nal
representations of a simply-laced quantum affine algebra, parametrized by arbitrary Dynkin quivers. 
For linearly oriented quivers of types $A$ and $D$, we show that these categories provide monoidal 
categorifications of cluster algebras of the same type. 
The proof is purely representation-theoretical, in the spirit of \cite{HL}.
\end{abstract}

\setcounter{tocdepth}{1}
{\footnotesize \tableofcontents}

\section{Introduction}


The theory of cluster algebras has received a lot of attention in the recent years
because of its numerous connections with many fields, in particular Lie theory
and quiver representations.

One important problem is to categorify cluster algebras.
In recent years, many examples of {\it additive} categorifications of cluster algebras have been
constructed. 
The concept of a {\it monoidal} categorification of a cluster algebra, which is quite different,
was introduced in 
\cite[Definition 2.1]{HL}. If a cluster algebra has a monoidal categorification,
we get informations on its structure (positivity, linear independence of cluster monomials).
Conversely, if a monoidal category is a monoidal categorification
of a cluster algebra of finite type, we can calculate the
factorization of any simple object as a tensor product of finitely many prime objects, 
as well as the composition factors of a tensor product of simple objects.

In \cite{HL} we have introduced a certain monoidal subcategory
$\mathcal{C}_1$ of the category $\mathcal{C}$ of finite-dimensional representations
of a simply-laced quantum affine algebra, and we have conjectured that $\mathcal{C}_1$ 
is a monoidal categorification of a cluster algebra of the same type. This conjecture was proved
in \cite{HL} for types $A$ and $D_4$, and in \cite{Ncl} for all $A, D, E$ types. 
The proof in \cite{HL} relies on representation theory, and on the well-developed combinatorics of
cluster algebras of finite type. Nakajima's proof is different
and uses additional geometric tools: a tensor category of perverse sheaves on quiver varieties,
and the Caldero-Chapoton formula for cluster variables.

The categories $\CC_1$ of \cite{HL} are associated with bipartite Dynkin quivers.
In this note, we introduce monoidal subcategories $\mathcal{C}_\xi$ of $\mathcal{C}$ 
associated with arbitrary Dynkin quivers. 
For types $A$ and $D$, we show that the categories $\CC_\xi$ corresponding to
linearly oriented quivers provide new monoidal categorifications
of cluster algebras of the same type. The proof is similar to \cite{HL}. However, the main calculations are 
much simpler because, for these choices of $\xi$, the irreducibility criterion for products of prime representations 
is more accessible than for the categories $\mathcal{C}_1$. This is why we can 
also treat in this note the cases $D_n$ ($n\geq 5$).

In his PhD thesis, Fan Qin \cite{Q} has recently generalized the geometric approach of Nakajima
(partly in collaboration with Kimura), and obtained monoidal categorifications of cluster algebras associated with an arbitrary
acyclic quiver (not necessarily bipartite) using perverse sheaves on quiver varieties.

\medskip

{\bf Acknowledgments :} The first author would like to thank A. Zelevinsky for explaining the results in \cite{yz, y}.
The authors are grateful to the referee for useful comments.

\section{Cluster algebras and their monoidal categorifications}\label{sect3}

We refer to \cite{FZSurv, K} for excellent surveys on cluster algebras.

\subsection{}\label{defcluster}
Let $0\le n < r$ be some fixed integers.
If $\widetilde{B} = (b_{ij})$ is an $r \times (r-n)$-matrix
with integer entries, then
the {\it principal part} $B$ of 
$\widetilde{B}$ is the square matrix obtained from 
$\widetilde{B}$ by deleting the last $n$ rows.
Given some $k \in [1,r-n]$ define a new $r \times (r-n)$-matrix 
$\mu_k(\widetilde{B}) = (b_{ij}')$ by
\begin{equation}
b'_{ij}=
\left\{
\matrix{
-b_{ij}\hfill & \mbox{if $i=k$ or $j=k$},\cr 
b_{ij} + {\displaystyle\frac{|b_{ik}|b_{kj} + b_{ik}|b_{kj}|}{2}} & \mbox{otherwise},
}
\right.
\end{equation}
where $i \in [1,r]$ and $j \in [1,r-n]$.
One calls $\mu_k(\widetilde{B})$ the {\it mutation} of the matrix $\widetilde{B}$
in direction $k$.
If $\widetilde{B}$ is an integer matrix whose principal part is
skew-symmetric, then it is 
easy to check that $\mu_k(\widetilde{B})$ is also an integer matrix 
with skew-symmetric principal part.
We will assume from now on that $\widetilde{B}$ has skew-symmetric principal part. 
In this case, one can equivalently encode $\widetilde{B}$ by a quiver $\Gamma$ with vertex set
$\{1,\ldots,r\}$ and with $b_{ij}$ arrows from $j$ to $i$ if $b_{ij}>0$
and $-b_{ij}$ arrows from $i$ to $j$ if $b_{ij}<0$.

Now Fomin and Zelevinsky define a cluster algebra
$\A(\widetilde{B})$ as follows.
Let $\F = \Q(x_1,\ldots,x_r)$ be the field of rational
functions in $r$ commuting indeterminates 
$\xx = (x_1,\ldots,x_r)$. 
One calls $(\xx,\widetilde{B})$ the {\it initial seed} of~$\A(\widetilde{B})$.
For $1 \le k \le r-n$ define 
\begin{equation}\label{mutationformula}
x_k^* = 
\frac{\prod_{b_{ik}> 0} x_i^{b_{ik}} + \prod_{b_{ik}< 0} x_i^{-b_{ik}}}{x_k}.
\end{equation}
The pair 
$(\mu_k(\xx),\mu_k(\widetilde{B}))$, 
where
$\mu_k(\xx)$ is obtained from $\xx$ by replacing $x_k$ by 
$x_k^*$,
is the {\it mutation} of the seed $(\xx,\widetilde{B})$ in direction $k$. 
One can iterate this procedure and obtain new seeds by mutating 
$(\mu_k(\xx),\mu_k(\widetilde{B}))$ in any direction 
$l\in [1,r-n]$. 
Let $\SS$ denote the set of all seeds obtained from $(\xx,\widetilde{B})$ 
by any finite sequence of mutations.
Each seed of $\SS$ consists of an $r$-tuple of elements of $\F$
called a {\em cluster}, and of a matrix.
The elements of a cluster are its {\em  cluster variables}.
One does not mutate the last $n$ elements of a cluster; they are
called {\em frozen variables} and belong to every cluster.
We then define the {\em cluster algebra}
$\A(\widetilde{B})$ as
the subring of $\F$ generated by all the cluster variables of 
the seeds of $\SS$.
A {\em cluster monomial} is a monomial in the cluster variables 
of a {\em single} cluster. Two cluster variables are said to be {\it compatible} if they occur in the same cluster.

The first important result of the theory 
is that every cluster variable $z$
of $\A(\widetilde{B})$ is a Laurent polynomial in
$\xx$ with
coefficients in $\Z$.
It is conjectured that the coefficients are positive.

The second main result is the classification of {\em cluster algebras
of finite type}, \ie with finitely many different cluster variables.
Fomin and Zelevinsky proved that this happens if and only if there
exists a seed $(\zz, \widetilde{C})$ such that the quiver
attached to the principal part of $\widetilde{C}$
is a Dynkin quiver (that is, an arbitrary orientation of a Dynkin diagram of
type $A, D, E$).

In \cite{FZ4}, Fomin and Zelevinsky have shown that the cluster
variables of a cluster algebra $\A$ have a nice expression in terms of certain
polynomials called the {\it $F$-polynomials}. In type $A$ and $D$, explicit formulas for $F$-polynomials are known.

\subsection{}
The concept of a monoidal categorification of a cluster algebra was introduced in 
\cite[Definition 2.1]{HL}. We say that a simple object $S$ of a monoidal category is 
{\em real} if $S\otimes S$ is simple.

\begin{definition}\label{defmoncat}
Let $\A$ be a cluster algebra and let $\M$ be an abelian monoidal category.
We say that $\M$ is a monoidal categorification of $\A$
if there is an isomorphism between $\A$ and the Grothendieck ring of $\M$ 
such that the cluster monomials of $\A$ are the classes of all the 
real simple objects of $\M$.
\end{definition}

A non trivial simple object $S$ of $\M$ is {\em prime}
if there exists no non trivial factorization $S\cong S_1 \otimes S_2$.
By \cite[Section 8.2]{GLS}, the cluster variables of $\A$ 
are the classes of all the real prime simple objects of $\M$. 
So Definition \ref{defmoncat} coincides with the definition in \cite{HL}.

As an application, we get information on the cluster algebra, as shown by the following result.
\begin{proposition}\cite{HL}\label{positivity}
If a cluster algebra $\A$ has a monoidal categorification, then
\begin{itemize}
 \item[{\rm (i)}]
every cluster variable of $\A$ has a Laurent expansion 
with positive coefficients with respect to any cluster;
\item[{\rm (ii)}]
the cluster monomials of $\A$ are linearly independent.
\end{itemize}
\end{proposition}
Assertion (ii) can also be proved by using additive categorification, see the recent \cite{CKLP}.

Conversely, if $\M$ is a monoidal categorification
of a finite type cluster algebra, we can calculate the
factorization of any simple object of $\M$ as a tensor product of finitely many prime objects, 
as well as the composition factors of a tensor product of simple objects of $\M$. Moreover,
every simple object in $\M$ is real.

\section{Categories of finite-dimensional representations of $U_q(L \g)$}\label{sect2}

For recent surveys on the representation theory of quantum loop algebras,
we invite the reader to consult \cite{CH} or \cite{L}.

\subsection{}
 Let $\g$ be a simple Lie algebra of type $A, D, E$.
We denote by $I$ the set of vertices of its Dynkin diagram,
and we put $n = |I|$. 
The \emph{Cartan matrix} of $\g$ is the $I\times I$ matrix $C = (C_{ij})_{i,j\in I}$.
We denote by $\a_i\ (i\in I)$ and $\varpi_i\ (i\in I)$ the simple roots and fundamental
weights of $\g$, respectively.

Let $\xi\colon I \to \Z$ be a \emph{height function}, that is $|\xi_j - \xi_i| = 1$ if $C_{ij} = -1$.
It induces an orientation $Q$ of the Dynkin diagram of $\g$ such that we have an arrow
$i\to j$ if $C_{ij} = -1$ and $\xi_j = \xi_i-1$. Define
\[
\hI := \{(i,p)\in I\times\Z \mid p-\xi_i \in 2\Z\}.
\]

\subsection{}
Let $L\g$ be the loop algebra attached to $\g$, and let $U_q(L\g)$ be the
associated quantum enveloping algebra. 
We assume that the deformation parameter $q\in\C^*$ is not a root of unity.

The simple finite-dimensional irreducible $U_q(L\g)$-modules (of type 1)
are usually labeled by Drinfeld polynomials. Here we shall use an 
alternative labeling by dominant monomials (see  \cite{FR}).
Moreover, as in \cite{HL}, we shall restrict our attention to a certain tensor subcategory 
$\CC_\Z$ of the category of finite-dimensional $U_q(L\g)$-modules. 
The simple modules in $\CC_\Z$ are labeled by the dominant monomials
in $\Y = \Z\left[Y_{i,p}^{\pm 1}\mid (i,p) \in \hI\right]$, that is monomials
$m = \prod_{(i,p)\in \hI} Y_{i,p}^{u_{i,p}(m)}$
such that $u_{i,p}(m) \ge 0$ for every $(i,p)\in \hI$.

We shall denote by $L(m)$ the simple module labeled by the dominant monomial $m$. 

By \cite{FR}, every object $M$ in $\CC_\Z$ has a $q$-character $\chi_q(M)\in\Y$. 
These $q$-characters generate a commutative ring $\K$
isomorphic to the Grothen\-dieck ring of $\CC_\Z$.

By \cite{FR, FM}, we have $\chi_q(L(m))\in m\Z[A_{i,p+1}^{-1}]_{(i,p)\in\hI}$ where
for $(i,p)\in\hI$ we denote
$$A_{i,p+1} = Y_{i,p}Y_{i,p+2}\prod_{j\in I, C_{ij} = -1}Y_{j,p+1}^{-1}\in \Y.$$
In particular, an element in $\K$ is characterized by the multiplicity of its dominant monomials.
When $m$ is the only dominant monomial occurring in $\chi\in\Y$, $\chi$
is said to be \emph{minuscule}. We say that $M$ is minuscule
if $\chi_q(M)$ is minuscule. This implies that $M$ is simple.

\subsection{}

Define
\[
\hI_\xi := \{(i,\xi_i)\mid i\in I\}\cup \{(i,\xi_i + 2)\mid i\in I\} \subset \hI, 
\]
and let $\Y_\xi$ be the subring
of $\Y$ generated by the variables $Y_{i,p}\ ((i,p)\in \hI_\xi)$.

\begin{definition} $\CC_\xi$ is the full subcategory of $\CC_\Z$ whose
objects have all their composition factors of the form $L(m)$
where $m$ is a dominant monomial in $\Y_\xi$.
\end{definition}
When $Q$ is a sink-source orientation, we recover the subcategories $\CC_1$ introduced in \cite{HL}.
Since $\hI_\xi$ is a ``convex slice'' of $\hI$, we get as in \cite[Lemma 5.8]{HL2} :

\begin{lemma}\label{tensorCQ}
$\CC_\xi$ is closed under tensor products, hence is a monoidal subcategory
of $\CC_\Z$. 
\end{lemma}

We denote by $\K_\xi$ the subring of $\K$ spanned by 
the $q$-characters $\chi_q(L(m))$ of the simple objects $L(m)$ in $\CC_\xi$.
Then $\K_\xi$ is isomorphic to the Grothendieck ring $\mathcal{R}_\xi$ of $\CC_\xi$.
Note that this ring is a polynomial ring over $\Z$ with generators the classes
of the $2n$ fundamental modules 
\[
L(Y_{i,\xi_i}),\quad L(Y_{i,\xi_i+2}),\qquad (1\le i\le n).
\]
The $q$-character of a simple object $L(m)$ of $\CC_\xi$ contains in general many
monomials $m'$ which do not belong to $\Y_\xi$. By discarding these monomials we
obtain a \emph{truncated $q$-charac\-ter} \cite{HL}. 
We shall denote by $\bchi_q(L(m))$ the truncated 
$q$-character of $L(m)$.
One checks that for a simple object $L(m)$ of $\CC_\xi$,
all the dominant monomials occurring in $\chi_q(L(m))$
belong to the truncated $q$-character $\bchi_q(L(m))$
(the proof is similar to that of \cite{HL} for the category $\CC_1$, as for the proof 
of Lemma~\ref{tensorCQ} above).
Therefore the truncation map $\chi_q(L(m))\mapsto \bchi_q(L(m))$
extends to an injective algebra homomorphism from $\K_\xi$ to $\Y_\xi$.

It is sometimes convenient to renormalize the (truncated) $q$-character
of $L(m)$ by dividing it by~$m$, so that its leading term becomes $1$.
The element of $\Y$ thus obtained is called a \emph{renormalized (truncated) $q$-character}. 

Define a partial ordering $\preceq$ on $\Y$ by $\chi\preceq \chi'$ if $\chi'-\chi$ is 
an $\N$-linear combination of monomials.
In particular, we have $\bchi_q(M)\preceq \chi_q(M)$ for $M$ in $\mathcal{C}_\xi$. 

\subsection{}\label{tools} Let $J\subset I$ and $\g_J\subset \g$ be the corresponding Lie subalgebra. Let $\hI_J = \hI\cap (J\times \Z)$. For $m$ a monomial, let 
$m_J = \prod_{(i,p)\in \hI_J} Y_{i,p}^{u_{i,p}(m)}$. If $m_J$
is dominant, one says that $m$ is $J$-dominant. In this case, let $L_J(m)$
be the sum (with multiplicities) of the monomials $m'$ occurring in $mm_J^{-1}\chi_q(L(m_J))$
such that $m(m')^{-1}$ is a product of $A_{i,p+1}^{-1}$, $(i,p)\in \hI_J$.
The image of $L_J(m)$ in $\Z[Y_{i,p}]_{(i,p)\in \hI_J}$, obtained by sending the $Y_{i,p}$ to $1$ if $(i,p)\notin \hI_J$, is the $q$-character of the simple $U_q(L\g_J)$ labeled by $m_J$ \cite[Lemma 5.9]{h6}. In particular 
we have the following:
\begin{lemma}\label{fact} 
Let $m$ and $m'$ be two dominant monomials such that
$L(m)\otimes L(m')$ is simple. 
Then $L_J(m) L_J(m') = L_J(mm')$.
\end{lemma}
For $m$ a dominant monomial one has a decomposition \cite[Proposition 3.1]{H1}
\begin{equation}\label{decompj}L(m) = \sum_{m'}  \lambda_J(m') L_J(m')\end{equation}
where the sum runs over $J$-dominant monomials $m'$. The $\lambda_J(m')\in\N$ are unique. This corresponds to the decomposition of $L(m)$ in the Grothendieck ring of $U_q(L\g_J)$-modules. 
This decomposition gives an inductive process to construct monomials occurring
in $\chi_q(L(m))$. Let us start with $m_0 = m$. Then the monomials $m_1$ of $L_J(m_0)$
occur in $\chi_q(L(m))$. If $m_1$ is $J_1$-dominant ($J_1\subset I$) and
if $L_{J_1}(m_1)$ occurs in the decomposition (\ref{decompj}), then the monomials
$m_2$ of $L_{J_1}(m_1)$ occur in $\chi_q(L(m))$, and we continue. See \cite[Remark 3.16]{h6} for details.

\subsection{}\label{estacat} In this note, we follow the proof of \cite{HL} to establish that 
for certain choices of $\xi$ the category 
$\mathcal{C}_\xi$ is a monoidal categorification of a cluster algebra $\mathcal{A}$. 
Let us recall the main steps (see \cite{HL} for details) :

(1) We define a family $\mathcal{P}$ of prime simple modules in $\mathcal{C}_\xi$ and
we label the cluster variables of an acyclic initial seed $\Sigma$ of $\mathcal{A}$
with a  subset of $\mathcal{P}$.

(2) We prove that the renormalized truncated $q$-characters of the representations of $\mathcal{P}$ 
coincide with the $F$-polynomials with respect to $\Sigma$ of 
all the cluster variables of $\mathcal{A}$.

(3) We prove an irreducibility criterion for tensor products of two representations in $\mathcal{P}$.

(4) By using the following general result, we factorize every simple module in $\mathcal{C}_\xi$ as a tensor product of representations in $\mathcal{P}$.
\begin{theorem}\label{irretens}\cite{H3} Let $S_1,\ldots, S_N$ be simple objects in $\mathcal{C}$. Then $S_1\otimes S_2\otimes \cdots \otimes S_N$
is simple if and only $S_i\otimes S_j$ is simple for any $1\leq i < j\leq N$.
\end{theorem}
In the next sections, we follow these steps for a good choice of $\xi$ in types $A$ and $D$. 
We conjecture that for arbitrary choices of $\xi$ and for every type $A, D, E$, $\mathcal{C}_\xi$ is the monoidal
catagorification of a cluster algebra of the same type. For type $A$, this can be proved in the same way as explained
in Remark \ref{othero} (b). For other types, this can be probably established by using the methods in \cite{Ncl}.

\section{Type $A$} 

\subsection{}
Let $\mathcal{A}$ be a cluster algebra of type $A_n$ in the Fomin-Zelevinsky classification.
As is well-known, the combinatorics of $\A$ is conveniently recorded in a regular polygon ${\mathbf P}$
with $n+3$ vertices labeled from $0$ to $n+2$, see \cite[\S 12.2]{FZ2}. 
Here, each cluster variable $x_{ab}\ (0\le a < b \le n+2)$ is labeled by the segment 
joining vertex $a$ to vertex $b$. The cluster variables $x_{ab}$ for which the segment 
$[a,b]$ is a side of the polygon are frozen. Moreover we specialize 
\[
 x_{01} = x_{n+1,\,n+2} = x_{0,\,n+2} =1.
\]
The exchange relations (Ptolemy relations) are of the form
\begin{equation}\label{Ptolemy}
x_{ac}x_{bd} = x_{ab}x_{cd} + x_{ad}x_{bc},\qquad (a<b<c<d). 
\end{equation}
The clusters of $\mathcal{A}$ correspond to the triangulations of ${\mathbf P}$.
The variables $x_{0i}\ (2\le i \le n+1)$ together with the $n$ frozen variables
$x_{i,\,i+1}\ (1\le i \le n)$ form a cluster, whose associated quiver is
\[
\matrix{x_{02}& \rightarrow & x_{03} & \rightarrow & x_{04} & \rightarrow &
 \cdots &  \rightarrow & x_{0,\,n+1}  \cr
         \downarrow &\nwarrow & \downarrow &\nwarrow& \downarrow &\nwarrow& &\nwarrow& \downarrow  \cr
 x_{12} &  & x_{23} & & x_{34} & & \cdots
   &   & x_{n,\,n+1}     
}
\]
Note that the principal part of this quiver (\ie the subquiver with vertices the 
non-frozen variables) is a quiver of type $A_n$ with linear orientation.
We denote by $\Sigma$ this particular seed of $\A$.

\subsection{}\label{sect4.2}
Let $\g$ be of type $A_n$. We will write $Y_{0,p} = Y_{n+1,p} = 1$ for $p\in\Z$.
We choose the height function 
\[
\xi(i) := i, \qquad (i\in I),
\]
corresponding to a quiver $Q$ of type $A_n$ with linear orientation.
We define the following family 
of irreducible representations in $\mathcal{C}_\xi$:
\[\mathcal{P} := \{L(i,j) := L(Y_{i,i}Y_{j,j + 2}) \mid 0\leq i\leq j\leq n+1\}.\]
The simple modules $L(i,j)$ are evaluation representations whose $q$-characters are known (see referen\-ces in \cite{CH}). 
In particular they are prime. We have $\bchi_q(L(0,j)) = Y_{j,j+2}$ and if $i\neq 0$ we have
$$\bchi_q(L(i,j)) = Y_{i,i}Y_{j,j+2}(1 + A_{i,i+1}^{-1} + (A_{i,i+1}A_{i+1,i+2})^{-1} 
+ \cdots + (A_{i,i+1}A_{i+1,i+2}\cdots A_{j-1,j})^{-1}).$$
Dividing both sides by $Y_{i,i}Y_{j,j+2}$ and setting $t_i := A_{i,i+1}^{-1}$, we see that
this formula for the renormalized truncated $q$-characters coincides with the formula for $F$-polynomials 
computed in \cite[Example 1.14]{yz}. 
It is easy to deduce from this that we have the following relations in $\mathcal{R}_\xi$ 
(also obtained in \cite{MY}) :
\begin{equation}\label{ident}[L(i,k)][L(j,l)]
= [L(i,l)][L(j,k)] +
[L(i,j-1)][L(k+1,l)]\quad\mbox{if}\quad0\leq i < j \leq k < l \leq n+1.\end{equation}
Therefore, comparing with (\ref{Ptolemy}), we see that the assignment 
\[
x_{ab} \mapsto [L(a,b-1)],\qquad (0\le a < b \le n+2)
\]
extends to an isomorphism from the cluster algebra $\A$ to the Grothendieck ring $\RR_\xi$.
This isomorphism maps the seed $\Sigma$ to
\[
\matrix{L(0,1)& \rightarrow & L(0,2) & \rightarrow &
 \cdots &  \rightarrow & L(0,n)  \cr
         \downarrow &\nwarrow & \downarrow &\nwarrow&  &\nwarrow& \downarrow  \cr
 L(1,1) &  & L(2,2) & & \cdots
   &   & L(n,n)     
}
\]
where the $L(i,i)\ (1\leq i\leq n)$ correspond to frozen variables.

We say that $(i,k)$ and $(j,l)$ are \emph{crossing} if and only if $i < j\leq k < l$ or $j < i\leq l < k$.
Otherwise, we say that $(i,k)$ and $(j,l)$ are \emph{noncrossing}.
The next proposition is similar to the classical irreducibility criterion
for prime representations of $U_q(L\mathfrak{sl}_2)$,
except that here, spectral parameters are replaced by nodes of the Dynkin diagram.

\begin{proposition}\label{tensa} The module $L(i,j)\otimes L(k,l)$ is simple if and only if 
$(i,j)$ and $(k,l)$ are noncrossing.
\end{proposition}

\begin{proof} The ``only if'' part follows from (\ref{ident}). We prove the ``if'' part. Let $M = Y_{i,i}Y_{j,j + 2}Y_{k,k}Y_{l,l + 2}$. We have
$\bchi_q(L(M))\preceq \chi = \bchi_q(L(i,j)\otimes L(k,l))$. 
We prove the other inequality. By symmetry, we are reduced to the following two cases:

\smallskip\noindent
(a)\ \ if $j < k$ or ($k = 0$ and $i, j\leq l$) or ($1\leq k\leq i, j = l$), then $\chi$ 
contains a unique dominant monomial, namely $M$, 
so $L(i,j)\otimes L(k,l)$ is simple.

\smallskip\noindent
(b)\ \ if $1\leq k\leq i\leq j < l$, then $\chi$ contains exactly two dominant monomials, namely $M$ and 
$$M'= M(A_{k,k+1}A_{k+1,k+2}\cdots A_{j,j+1})^{-1}.$$ 
So it suffices to prove that $M'$ occurs in $\bchi(L(M))$. First, by \S\ref{tools}, the monomial
$$M'' = M(A_{k,k+1}A_{k+1,k+2}\cdots A_{i-1,i})^{-1}$$ occurs in $\bchi(L(M))$. Hence
$L_J(M'')$ occurs in the decomposition (\ref{decompj}) for $J=\{i, \ldots, n\}$.
But $L(Y_{i,-i} Y_{j,-j-2})\otimes L(Y_{i,-i})$ is minuscule and simple. Hence, by \cite[Corollary 4.11]{miniaff}, 
the tensor product $L(Y_{i,i} Y_{j,j+2})\otimes L(Y_{i,i})$ is simple, isomorphic to $L(Y_{i,i}^2 Y_{j,j+2})$. So $Y_{i,i}^2 Y_{j,j+2}(A_{i,i+1}\cdots A_{j,j+1})^{-1}$ occurs in $\bchi(L(Y_{i,i}^2 Y_{j,j+2}))$ and $M'$ occurs in $L_J(M'')$.
\cqfd
\end{proof}
Therefore, as explained in \S\ref{estacat}, we get the following:
\begin{theorem}\label{TH1} 
$\mathcal{C}_\xi$ is a monoidal categorification of the cluster algebra $\A$ of type $A_n$.
\end{theorem}

\begin{remark}\label{othero}
{\rm 
(a)\ \ It follows from Theorem~\ref{TH1} that when $\xi_i = i$, every
simple module in $\mathcal{C}_\xi$ can be factorized as a tensor product of evaluation representations.

\smallskip
(b)\ \ 
For an arbitrary $\xi$, a theorem similar to Theorem~\ref{TH1} can be proved in an analog but 
slightly more complicated way. 
A subset $J= [i,j]\subset I$  $(1\leq i \leq j \leq n)$ has a natural
orientation induced by $\xi$. Let $J_+$ (resp. $J_-$) be the set of sources (resp. sinks) of $J$. 
The prime objects in $\mathcal{C}_\xi$ are the simple modules
$$L(J) := L\left(\prod_{k\in J_-}Y_{k,\xi_k}\prod_{k\in J_+}Y_{k,\xi_k + 2}\right)\quad ,
\quad L(i) := L(Y_{i,\xi_i})\quad , \quad L'(i) := L(Y_{i,\xi_i + 2}).$$
Note that $L(J)$ is not an evaluation representation if $J$ has several sources or several sinks. 

\smallskip
(c)\ \
Different choices of $\xi$ yield different subcategories $\mathcal{C}_\xi$.
These subcategories seem to be quite similar, but they are not
equivalent in general. For example, in type $A_3$, consider
the categories $\mathcal{C}_\xi$ with $\xi_i = i$ and
$\mathcal{C}_{\phi}$ with $\phi_1 = 1$, $\phi_2 = 2$, $\phi_3 = 1$.
Both categories are monoidal categorifications of a cluster algebra
of type $A_3$ with $3$ coefficients. The category $\mathcal{C}_\phi$ was
studied in \cite{HL}. In particular, we the have following relation in the Grothendieck ring of $\mathcal{C}_\phi$ :
$$[L(Y_{1,1}Y_{2,4}Y_{3,1})][L(Y_{2,2})]
= [L(Y_{1,1})][L(Y_{3,1})][L(Y_{2,2}Y_{2,4})]
+ [L(Y_{1,1}Y_{1,3})][L(Y_{3,1}Y_{3,3})].$$
But by (\ref{ident}), in the Grothendieck ring of $\mathcal{C}_\xi$, a simple constituent of the 
tensor product of two simple prime representations can be factorized as a tensor product of at most $2$ non trivial
representations. Hence, $\mathcal{C}_\xi$ and $\mathcal{C}_\phi$ are {\it not} equivalent.}
\end{remark}

\section{Type $D$} 

\subsection{}\label{clusterD}
Let $\mathcal{A}$ be a cluster algebra of type $D_n$ in the Fomin-Zelevinsky classification.
The clusters of $\mathcal{A}$ are now encoded by the \emph{centrally symmetric} triangulations of a regular 
polygon ${\mathbf P}$ with $2n$ vertices, labeled by $a=0,\ 1,\ldots, 2n-1$ \cite[\S 12.4]{FZ2}
(note that a more modern way to record the
combinatorics of a cluster algebra of type $D_n$ would be by means of a once-punctured $n$-gone and tagged arcs \cite{FST}).
A segment $[a,b]$ joining two vertices is called a \emph{diagonal} if it meets the interior of ${\mathbf P}$,
and a \emph{side} otherwise.
Let $\Theta$ be the $180^{\circ}$ rotation of~$\mathbf{P}$, and for a vertex~$a$,
write $\overline{a}=\Theta(a)$.
Each \emph{non frozen} cluster variable is labeled by a $\Theta$-orbit on the set of diagonals of~${\mathbf P}$.
More precisely, to each non trivial $\Theta$-orbit $([a,b],\,[\overline{a},\overline{b}])$ (with $b \not =\overline{a}$)
we attach a single cluster variable 
\[
x_{ab} = x_{\overline{a}\overline{b}}.
\]
But we associate with every $\Theta$-fixed diagonal $[a,\overline{a}]$ (or \emph{diameter}) two \emph{different} cluster variables
\[
x_{a\overline{a}} \not = x_{\widetilde{a\overline{a}}}.  
\]
We may think of $[a,\overline{a}]$ and $\widetilde{[a,\overline{a}}]$
as two different $\Theta$-orbits, supported on the same segment but with two different colors.
Given two $\Theta$-orbits, one of which at least being non trivial,
we say that they are \emph{noncrossing} if they
do not meet in the interior of~$\mathbf{P}$. 
We also declare that 
two $\Theta$-fixed diagonals are \emph{noncrossing} if and only if 
they have the same support or the same color.
A \emph{centrally symmetric triangulation} of $\mathbf{P}$ is then 
a maximal subset of pairwise noncrossing $\Theta$-orbits of diagonals.
Such a triangulation always consists of $n$
different $\Theta$-orbits.
For instance, for $n=4$, the following are two distinct triangulations
\[
\left\{([1,\overline{3}],\ [\overline{1},3]),\ 
([2,\overline{3}],\ [\overline{2},3]),\  
[3,\overline{3}],\ \widetilde{[3,\overline{3}]}\right\},
\qquad
\left\{([1,\overline{3}],\ [\overline{1},3]),\ 
([2,\overline{3}],\ [\overline{2},3]),\  
[3,\overline{3}], [2,\overline{2}]\right\}. 
\]
To the $\Theta$-orbits of
the sides $[a,b]$ of $\mathbf{P}$ we can also attach some frozen variables 
$x_{ab} = x_{\overline{a}\overline{b}}$. We specialize
\[
x_{01} = x_{\overline{n-1},\,0} = 1. 
\]
Our initial seed for the cluster algebra $\A$ will correspond
to the triangulation
\[
\left\{\Theta([a,\overline{n-1}]) \mid 1\le a \le n-2\right\}
\cup
\left\{[n-1,\overline{n-1}],\widetilde{[n-1,\overline{n-1}]}\right\}. 
\]
More precisely, it is described by the following
quiver
\[
\matrix{& & & & & &  &  & x_{\widetilde{n-1,\,\overline{n-1}}}&\leftarrow & f_{n-1}  \cr 
& & & & & & & & \uparrow &  \searrow& \cr 
x_{1,\,\overline{n-1}}& \rightarrow & x_{2,\,\overline{n-1}} & \rightarrow &
 \cdots&\rightarrow & x_{n-3,\,\overline{n-1}}&\rightarrow & x_{n-2,\,\overline{n-1}} &\leftarrow & x_{n-2,\,n-1} \cr
      \uparrow &\swarrow & \uparrow &\swarrow& &\swarrow&\uparrow &\swarrow& \downarrow  &\nearrow\cr
 x_{12} &  & x_{23} & & \cdots 
&  &x_{n-3,\,n-2}  &   & x_{n-1,\,\overline{n-1}} &\leftarrow & f_n
}
\]
where $f_n$ and $f_{n-1}$ are two additional frozen variables, which
can not be encoded by sides of $\mathbf{P}$.
The principal part of the quiver (obtained by removing the frozen vertices 
$x_{i,\,i+1}\ (1\le i\le n-2)$, $f_{n-1}$, $f_n$, and the arrows
incident to them) is a Dynkin quiver $Q$ of type $D_n$, hence $\mathcal{A}$ is 
indeed a cluster algebra of type $D_n$ in the Fomin-Zelevinsky classification. 

One can easily check that, because of this particular choice of frozen
variables, $\A$ belongs to the class of cluster algebras
studied in \cite{GLS1}. 
More precisely, let us label the vertices of $Q$ by $\{1,\ldots, n\}$
so that $x_{i,\,\overline{n-1}}$ lies at vertex $i$ for $i\le n-1$,
and $x_{\widetilde{n-1,\,\overline{n-1}}}$ lies at vertex $n$.
Then $\A$ is the same as the algebra attached in \cite{GLS1} to $Q$
and the Weyl group element 
\[
w= c^2 = (s_ns_{n-1}s_{n-2}\cdots s_1)^2.
\]
It follows from \cite[Theorem 16.1 (i)]{GLS1} that $\A$ is a polynomial ring in $2n$ generators.
These generators are the initial cluster variables 
\[
z_i:=x_{i,\,\overline{n-1}}\quad (1\le i\le n-1),\quad
z_n:=x_{\widetilde{n-1,\,\overline{n-1}}},
\]
together with the new cluster variables $z^\dag_{i}\ (1\le i\le n)$ produced by the sequence of mutations
\begin{equation}\label{seqmut}
 \mu_{n}\circ \mu_{n-1}\circ\mu_{n-2}\circ \cdots \circ \mu_{2}\circ \mu_{1}.
\end{equation}

Recall from \cite{FZ2} that, our initial cluster being fixed, 
the cluster variables of $\A$ also have a natural labelling by 
almost positive roots. The correspondence is as follows.
First, the $\Theta$-orbits of the initial triangulation are labeled
by negative simple roots:
\[
\Theta([i,\overline{n-1}]) \mapsto -\alpha_i,\quad  (1\le i \le n-2),\qquad
[n-1,\overline{n-1}] \mapsto -\a_{n-1},\qquad
[\widetilde{n-1,\overline{n-1}}] \mapsto -\a_n. 
\]
Any other $\Theta$-orbit $x$ is mapped to the positive root $\sum_i c_i\a_i$,
where the diagonals representing $x$ cross the diagonals representing
$-\a_i$ at $c_i$ pairs of centrally symmetric
points 
(counting an intersection of two diameters of different colors and support
as one such pair).

In \cite{yz, y}, a different labelling for the cluster variables is used.
First the choice of an
\emph{acyclic} initial seed is encoded by the choice of a Coxeter element $c$.
For our choice of initial seed, this Coxeter element is  
\[
c = s_ns_{n-1}s_{n-2}\cdots s_1.
\]
Next the cluster variables 
are labeled by weights of the form
\[
c^m \varpi_{i},\quad (i\in I,\ 0\le m \le h(i,c)),
\]
where $h(i,c)$ is the smallest integer such that $c^{h(i,c)}\varpi_i = w_0\varpi_i$.
The correspondence between the two labellings is as follows.
To the fundamental weight $\varpi_i$ corresponds $-\a_i$, 
and to the weight $c^m \varpi_{i}\ (m\ge 1)$ corresponds the positive root
$\b = c^{m-1}\varpi_i - c^m \varpi_{i}$. 

\begin{example}\label{exD4}
{\rm
We illustrate all these definitions in the case $n=4$. 
Here $\mathbf{P}$ is a regular octogon, with vertices labeled
by $0, 1, 2, 3, \overline{0}, \overline{1}, \overline{2}, \overline{3}$.
Our choice of initial triangulation is 
\[
\left\{([1,\overline{3}],\ [\overline{1},3]),\ 
([2,\overline{3}],\ [\overline{2},3]),\  
[3,\overline{3}],\ \widetilde{[3,\overline{3}]}\right\},
\] 
which corresponds to the Coxeter element $c = s_4s_3s_2s_1$.
The sixteen $\Theta$-orbits of diagonals (represented by one of their 
elements), and the corresponding
indexings by almost positive roots, and by weights, are given in
the table below:  
\[
\begin{array}{|c|c|c|}
\hline
[1,\overline{3}] & -\a_1 & \varpi_1\\[1mm]
[2,\overline{3}] & -\a_2 & \varpi_2\\[1mm]
\widetilde{[3,\overline{3}]} & -\a_3 & \varpi_3\\[1mm]
[3,\overline{3}] & -\a_4 & \varpi_4\\[1mm]
[0,2] & \a_1 & c^3\varpi_1  \\[1mm]
[1,3] & \a_2 & c^2\varpi_1  \\[1mm]
[2,\overline{2}] & \a_3 & c\varpi_3  \\[1mm]
\widetilde{[2,\overline{2}]} & \a_4 & c\varpi_4  \\[1mm]
\hline
\end{array}
\qquad\qquad
\begin{array}{|c|c|c|}
\hline
[0,3] & \a_1+\a_2 & c^3\varpi_2\\[1mm]
[1,\overline{1}] & \a_2+\a_3 & c^2\varpi_4\\[1mm]
\widetilde{[1,\overline{1}]} & \a_2 + \a_4 & c^2\varpi_3\\[1mm]
[0,\overline{0}] & \a_1+\a_2+\a_3 & c^3\varpi_3\\[1mm]
\widetilde{[0,\overline{0}]} & \a_1+\a_2+\a_4 & c^3\varpi_4  \\[1mm]
[2,\overline{1}] & \a_2+\a_3+\a_4 & c\varpi_2  \\[1mm]
[2,\overline{0}] & \a_1+\a_2+\a_3+\a_4 & c\varpi_1  \\[1mm]
[1,\overline{0}] & \a_1+2\a_2+\a_3+\a_4 & c^2\varpi_2  \\[1mm]
\hline
\end{array}
\]
}
\end{example}

\subsection{} Let $\g$ be of type $D_n$. We will write $Y_{0,p} = Y_{n+1,p} = 1$ for $p\in\Z$. 
We choose the height function 
$\xi_i = n - 1 - i$ if $i < n$ and $\xi_n = 0$. 
This induces a partial order $\preceq$ on $\{1,\ldots , n\}$ defined by
\[
i \prec j \quad \Longleftrightarrow \quad \xi_i < \xi_j.
\]
Note that $n-1$ and $n$ are not comparable for $\preceq$.
Moreover, for convenience, we extend this to $\{0,\ldots , n+1\}$
by declaring that $0$ is a maximal element and $n+1$ a minimal element
for $\preceq$.

We define the following family $\mathcal{P}$ of representations in $\mathcal{C}_\xi$:
\[
\begin{array}{ll}
L(i,j) = L(Y_{i,\xi_i}Y_{j,\xi_j + 2}), &  (n+1\preceq i\preceq j\preceq 0),\\[2mm]
L(i,j)^\dag 
= L(Y_{n,0}Y_{n-1,0}Y_{i,\xi_i+2}Y_{j,\xi_j + 2}),\quad& (n-2\preceq j \prec i\preceq 0).
\end{array}
\]
Since $\A$ and $\mathcal{R}_\xi$ are both polynomial rings over $\Z$ with $2n$ generators,
the assignment 
\[
z_i \mapsto [L(n+1,i)]=[L(Y_{i,\xi_i+2})], \quad 
z_i^\dag \mapsto [L(i,0)]=[L(Y_{i,\xi_i})], \qquad (1\le i \le n),  
\]
extends to a ring isomorphism $\iota\colon\A \isom \mathcal{R}_\xi$.
Thus $\mathcal{R}_\xi$ is endowed with the structure of a cluster algebra.
Moreover, using the $T$-system equations for calculating  
the products 
\[
[L(Y_{i,\xi_i})][L(Y_{i,\xi_i+2})] = z_iz_i^\dag,
\]
and comparing them with the exchange relations involved in the sequence of
mutations (\ref{seqmut}), we can easily check that the frozen 
variables of $\A$ are mapped by $\iota$ to the classes $[L(i,i)]=[L(Y_{i,\xi_i}Y_{i,\xi_i+2})]$.
More precisely,
\[
\iota(f_{n-1})=[L(n-1,n-1)],\quad
\iota(f_{n})=[L(n,n)],\quad
\iota(x_{i,i+1}) = [L(i,i)], \quad (1\le i \le n-2).
\]
Therefore $\iota$ maps the initial seed of $\A$ to
\[
\matrix{& & & & & &     L(n+1,\,n-1)&\leftarrow &L(n-1,\,n-1)  \cr 
& & & & & &   \uparrow &  \searrow& \cr 
L(n+1,1)& \rightarrow & L(n+1,2) & \rightarrow &
 \cdots  & \rightarrow & L(n+1,\,n-2) &\leftarrow &L(n-2,\,n-2) \cr
         \uparrow &\swarrow & \uparrow &\swarrow&  &\swarrow& \downarrow  &\nearrow\cr
 L(1,1) &  & L(2,2) & & \cdots 
     &  & L(n+1,\,n)&\leftarrow &L(n,\,n)
}
\]

\subsection{} Let us compute the truncated $q$-characters of the representations in $\mathcal{P}$. As in  \S\ref{sect4.2}, the modules $L(i,j)$ are prime minimal affinizations. We have 
\[
\begin{array}{llll}
\bchi_q(L(n+1,j)) &=& Y_{j,\xi_j + 2},&\\[2mm]
\bchi_q(L(i,j)) &=& Y_{i,\xi_i}Y_{j,\xi_j + 2} ( 1 + A_{i,\xi_i+1}^{-1} +\cdots + (A_{i,\xi_i+1}\cdots A_{j-1,\xi_j})^{-1}),& (n-1\preceq i),\\[2mm]
\bchi_q(L(n,j)) &= &Y_{n,0}Y_{j,\xi_j + 2} ( 1 + A_{n,1}^{-1} \chi_j),&(0\leq j\leq n-2),
\end{array}
\]
where $\chi_j := 1 + A_{n-2,2}^{-1} + \cdots + (A_{n-2,2}\cdots A_{j+1,\xi_j})^{-1}$.
In general the $L(i,j)^\dag$ are not minimal affi\-ni\-zations. However, we have:

\begin{lemma}\label{formu} For $n-2\preceq j \prec i\preceq 0$, the representation $L(i,j)^\dag$ is prime and 
$$\bchi_q(L(i,j)^\dag)) = Y_{n,0}Y_{n-1,0}Y_{i,\xi_i+2}Y_{j,\xi_j + 2}(1 + (A_{n-1,1}^{-1} + A_{n,1}^{-1})\chi_j + A_{n-1,1}^{-1}A_{n,1}^{-1}\chi_i\chi_j).$$
\end{lemma}

\begin{proof} As $\bchi_q(L(i,j)^\dag)\preceq \bchi_q(L(n,j)\otimes 
L(n-1,i))$ and $\bchi_q(L(i,j)^\dag)\preceq \bchi_q(L(n,i)\otimes 
L(n-1,j))$ there are $A, B \preceq \chi_j$ and
$C\preceq \chi_i\chi_j$ such that
$$\bchi_q(L(i,j)^\dag) = Y_{n,0}Y_{n-1,0}Y_{i,\xi_i+2}Y_{j,\xi_j + 2}(1 + 
A_{n-1,1}^{-1} A + A_{n,1}^{-1} B + A_{n,1}^{-1}A_{n-1,1}^{-1}C).$$
From Proposition \ref{tensa} with $J = \{1,\cdots, n-1\}$, we have
$$Y_{n,0}L_J(Y_{n-1,0}Y_{j,n-j+1}) L_J(Y_{i,n-i+1}) = 
L_J(Y_{n,0}Y_{n-1,0}Y_{i,n-i+1} Y_{j,n-j+1}).$$
Hence, by \S\ref{tools}, we have $A = \chi_j$. The proof that $B = \chi_j$ 
is analog. Similarly, from Proposition \ref{tensa} with $J = \{1,\cdots, 
n-2\}$, we have
$$L_J(Y_{n-2,1}Y_{i,n-i+1}) L_J(Y_{n-2,1}Y_{j,n-j+1}) = 
L_J(Y_{n-2,1}^2Y_{i,n-i+1}Y_{j,n-j+1}).$$
So
$$C = 
(Y_{n-2,1}^2Y_{i,n-i+1}Y_{j,n-j+1})^{-1}\bchi_q(L(Y_{n-2,1}^2Y_{i,n-i+1}Y_{j,n-j+1})) 
= \chi_i\chi_j.$$
This explicit formula shows that $\bchi_q(L(i,j)^\dag)$ can not be 
factorized and so $L(i,j)^\dag$ is prime.
\cqfd
\end{proof}

Let $\P':=\P\setminus\{L(i,i)\mid 1\le i \le n\}$. 
We introduce the following bijection between the non frozen cluster variables of $\A$ and
the representations in $\P'$.


\[
\begin{array}{llll}
x_{ij}&\mapsto& L(j-1,i), & (0\leq i\leq j-2\leq n - 3), \\[1mm]
x_{i\,\overline{j}} &\mapsto & L(j,i)^\dag, & (0\leq j < i \leq n-2), \\[1mm]
x_{i\,\overline{i}} &\mapsto & L(n-1,i), &   (0\leq i\leq n-2), \\[1mm]
x_{\widetilde{i\,\overline{i}}} &\mapsto & L(n,i), & (0\leq i\leq n-2), \\[1mm]
x_{i,\,\overline{n-1}} & \mapsto & L(n+1,i), & (1\leq i\leq n-2),  \\[1mm]
x_{\widetilde{n-1,\,\overline{n-1}}} & \mapsto & L(n+1,n-1),\\[1mm]  
x_{n-1,\,\overline{n-1}} & \mapsto & L(n+1,n). 
\end{array}
\]
One can check that under this correspondence, the renormalized truncated $q$-characters 
for the representations in $\P'$ coincide with the $F$-polynomials 
of the cluster variables of $\A$ calculated in \cite{yz, y}. 
One then deduces that this bijection is the restriction of the ring automorphism
$\iota$ to the set of non frozen cluster variables.

\begin{example}
{\rm
We continue Example~\ref{exD4}.
The table below gives the list of cluster variables of $\A$
together with the corresponding
representations of $\P'$ and their
truncated $q$-characters.
Here $t_i = A_{i,\xi_i +1}^{-1}$.

\[
\begin{array}{|l|l|l|}
\hline
x_{02}&L(1,0)  &Y_{1,2}(1 + t_1)\\[1mm]
x_{03}&L(2,0)  &Y_{2,1}(1 + t_2 + t_2t_1)\\[1mm]
x_{13}&L(2,1)  & Y_{1,4}Y_{2,1}(1 + t_2)\\[1mm]
x_{1\overline{0}}&L(0,1)^\dag & Y_{1,4}Y_{3,0}Y_{4,0}(1 + t_3 + t_3t_2 + t_4 + t_4t_2 + t_3t_4 + 2 t_3t_4t_2 + t_3t_4t_2^2 + t_3t_4t_2t_1 + t_3t_4t_2^2t_1)\\[1mm]
x_{2\overline{0}}&L(0,2)^\dag &Y_{2,3}Y_{3,0}Y_{4,0}(1 + t_3 + t_4 + t_3t_4 + t_3t_4t_2 + t_3t_4t_2t_1)\\[1mm]
x_{2,\overline{1}}&L(1,2)^\dag & Y_{1,4}Y_{2,3}Y_{3,0}Y_{4,0}(1 + t_3 + t_4 + t_3t_4 + t_3t_4t_2)\\[1mm]
x_{0\overline{0}}&L(3,0) & Y_{3,0}(1 + t_3 + t_3t_2 + t_3t_2t_1)\\[1mm]
x_{1\overline{1}}&L(3,1) &  Y_{1,4}Y_{3,0}(1 + t_3 + t_3t_2)\\[1mm]
x_{2\overline{2}}&L(3,2)& Y_{2,3}Y_{3,0}(1 + t_3)\\[1mm]
x_{\widetilde{0\overline{0}}}&L(4,0) & Y_{4,0}(1 + t_4 + t_4t_2 + t_4t_2t_1)\\[1mm]
x_{\widetilde{1\overline{1}}}&L(4,1) &  Y_{1,4}Y_{4,0}(1 + t_4 + t_4t_2)\\[1mm]
x_{\widetilde{2\overline{2}}}&L(4,2) &Y_{2,3}Y_{4,0}(1 + t_4)\\[1mm]
x_{1\overline{3}}&L(5,1) &Y_{1,4}\\[1mm] 
x_{2\overline{3}}&L(5,2) &Y_{2,3}\\[1mm] 
x_{\widetilde{3\overline{3}}}&L(5,3) &Y_{3,2}\\[1mm] 
x_{3\overline{3}}&L(5,4)&Y_{4,2}\\
\hline
\end{array}
\]
}
\end{example}

\subsection{} 
We now describe which tensor products of representations of $\P$ are simple.

\begin{proposition} \label{lastprop}
We have the following :
\begin{itemize}
\item[(a)] Suppose $\{i,k\} \neq \{n-1,n\}$. 
Then $L(i,j)\otimes L(k,l)$ is not simple if and only if 
$i \prec k\preceq j \prec l$ or $k \prec i \preceq l \prec j$.

\item[(b)] Suppose $\{i,k\} = \{n-1,n\}$. 
Then $L(i,j)\otimes L(k,l)$ is simple if and only if $j = l$ or $i = j$ or $k = l$.

\item[(c)] Suppose $j\prec i$ and $l\prec k$. 
Then $L(i,j)^\dag\otimes L(k,l)^\dag$ is simple if and only if 
$j\preceq l \prec k\preceq i$ or $l\preceq j \prec i\preceq k$.

\item[(d)] Suppose $i\succeq n-2$ and $l\prec k$. 
Then $L(i,j)\otimes L(k,l)^\dag$ is simple if and only if 
$i = j$ or $i\prec j\preceq l\prec k$ or $l\prec k\prec i\prec j$ or $l\prec i \prec j\preceq k$.

\item[(e)] Suppose $i\prec n-2$ and $l \prec k$.
Then $L(i,j)\otimes L(k,l)^\dag$ is simple if and only if 
$i = j$ or ($(i\neq n+1)$ and $l\preceq j\preceq k$) or ($i = n+1$ and $k\preceq j$).
\end{itemize}
\end{proposition}

\begin{proof} In each case, the proof of non simplicity follows from the identification of truncated $q$-characters with $F$-polynomials in the last section. So we treat only the proof of the simplicity.

(a) The irreducibility is proved as in type $A$, except for the tensor
product 
\[
L(n+1,n)\otimes L(n+1,n-1)
\] 
which is minuscule and so is simple. 

(b) If $n-2\preceq j = l$ or $i = j$ or $k = l$, $L(n,j)\otimes L(n-1,j)$ is minuscule and so is simple.

(c) By symmetry, we can assume $j\preceq l$. 
Suppose that $j\preceq l\prec k\preceq i$ and let us prove that $L(i,j)^\dag\otimes L(k,l)^\dag$
is simple. Let $M$ be its highest weight monomial. It suffices to prove that any dominant monomial $m$ occurring in 
$\bchi(L(i,j)^\dag)\bchi(L(k,l)^\dag)$ occurs in $\bchi_q(L(M))$.
If $A_{n-1,1}^{-1}$ or $A_{n,1}^{-1}$ is not a factor of $mM^{-1}$, this is proved as for type $A$.
If $A_{n-1,1}^{-2}$ is a factor of $mM^{-1}$, first from \S\ref{tools} 
$M A_{n-1,1}^{-2}$ occurs and $L_J(M A_{n-1,1}^{-2})$ occurs in the decomposition (\ref{decompj}) for $J = \{1,\cdots,n-2,n\}$. But from type $A$
$$L_J(M A_{n-1,1}^{-2}) = Y_{n-1,2}^{-2}L_J(Y_{n,0}Y_{n-2,1}Y_{i,n-i+1}Y_{j,n-j+1}) L_J(Y_{n,0}Y_{n-2,1}Y_{k,n-k+1}Y_{l,n-l+1})$$
and we can conclude by \S\ref{tools}.
This is analog if $A_{n,1}^{-2}$ is a factor. 
So we can assume that $A_{n,1}^{-1}$
and $A_{n-1,1}^{-1}$ are factors with power $1$. Then $m$ is one of the following monomials
$$M A_{n,1}^{-1}A_{n-1,1}^{-1}A_{n-2,2}^{-1}\cdots A_{j,n-j}^{-1}\quad \mbox{with multiplicity}\quad 5,$$
$$M A_{n,1}^{-1}A_{n-1,1}^{-1}A_{n-2,2}^{-1}\cdots A_{l,n-l}^{-1}\quad \mbox{with multiplicity}\quad 2,$$
$$M A_{n,1}^{-1}A_{n-1,1}^{-1}A_{n-2,2}^{-1}\cdots A_{k,n-k}^{-1}\quad \mbox{with multiplicity}\quad 1,$$
$$M A_{n,1}^{-1}A_{n-1,1}^{-1}A_{n-2,2}^{-2}\cdots A_{j,n-j}^{-2}A_{j+1,n-j-1}^{-1}\cdots A_{l,n-l}^{-1}\quad \mbox{with multiplicity}\quad 1.$$
Then we conclude as above. For example for the last monomial of the list, 
\[
M' := MA_{n,1}^{-1}A_{n-2,2}^{-1}\cdots A_{j,n-j}^{-1}
\] 
occurs in $L_{\{n,n-2,\cdots,j\}}(M)$ from type $A$.
Hence $M'A_{j,n-j}$ occurs in $L_{\{n-1,n-2,\cdots,j\}}(MA_{n,1}^{-1})$, but $M'$ does not. So $L_{\{n-1,n-2,\cdots, j\}}(M')$ occurs in the decomposition (\ref{decompj}). 
Since $M$ is a monomial in $L_{\{n-1,n-2,\cdots,1\}}(M')$, we get the result.

(d) and (e) : The proof is analog.


\cqfd
\end{proof}

Proposition~\ref{lastprop} implies that the tensor products of representations of $\P$ corresponding to 
compatible cluster variables are simple. Indeed, 
two cluster variables are compatible if and only if the corresponding diagonals in ${\mathbf P}$ do not cross (with the convention that diameters of the same color do not cross each other) \cite[\S 12.4]{FZ2}.
This coincides with the conditions of Proposition~\ref{lastprop}. 

\begin{example}
{\rm
We continue Example~\ref{exD4}. 
The following table lists the compatible pairs of non frozen variables of~$\A$, and indicates
in which case of Proposition~\ref{lastprop} the corresponding pairs of simple modules fall.

\[
\begin{array}{|c|c||c|c||c|c||c|c|}
\hline
(x_{02},\,x_{03})& (a) &(x_{02},\,x_{0\overline{0}})& (a) & 
(x_{02},\,x_{\widetilde{0\overline{0}}})& (a) &(x_{02},\,x_{2\overline{2}})& (a) \\[1mm]
(x_{02},\,x_{\widetilde{2\overline{2}}}) & (a) & (x_{02},\,x_{2\overline{3}})& (a) & 
(x_{02},\,x_{3\overline{3}})& (a) &(x_{02},\,x_{\widetilde{3\overline{3}}})& (a) \\[1mm]
(x_{03},\,x_{0\overline{0}}) & (a) & (x_{03},\,x_{\widetilde{0\overline{0}}})& (a) &
(x_{03},\,x_{13}) & (a) & (x_{03},\,x_{3\overline{3}})& (a) \\[1mm]
(x_{03},\,x_{\widetilde{3\overline{3}}})  & (a) & (x_{0\overline{0}},\,x_{13}) & (a) & 
(x_{0\overline{0}},\,x_{1\overline{1}}) & (a) &(x_{0\overline{0}},\,x_{2\overline{2}})& (a) \\[1mm]
(x_{0\overline{0}},\,x_{3\overline{3}})& (a) &   (x_{\widetilde{0\overline{0}}},\,x_{13}) & (a) &
(x_{\widetilde{0\overline{0}}},\,x_{\widetilde{1\overline{1}}}) & (a) & (x_{\widetilde{0\overline{0}}},\,x_{\widetilde{2\overline{2}}}) & (a) \\[1mm]
(x_{\widetilde{0\overline{0}}},\,x_{\widetilde{3\overline{3}}}) & (a) &  (x_{13},\,x_{1\overline{1}}) & (a) & 
(x_{13},\,x_{\widetilde{1\overline{1}}}) & (a) & (x_{13},\,x_{1\overline{3}}) & (a) \\[1mm]
(x_{13},\,x_{3\overline{3}}) & (a) & (x_{13},\,x_{\widetilde{3\overline{3}}}) & (a) &
(x_{1\overline{1}},\,x_{1\overline{3}}) & (a) & (x_{1\overline{1}},\,x_{2\overline{2}}) & (a) \\[1mm]
(x_{1\overline{1}},\,x_{3\overline{3}}) & (a) &  (x_{\widetilde{1\overline{1}}},\,x_{1\overline{3}}) & (a) &
(x_{\widetilde{1\overline{1}}},\,x_{\widetilde{2\overline{2}}}) & (a) & (x_{\widetilde{1\overline{1}}},\,x_{\widetilde{3\overline{3}}}) & (a) \\[1mm]
(x_{1\overline{3}},\,x_{2\overline{2}}) & (a) & (x_{1\overline{3}},\,x_{\widetilde{2\overline{2}}}) & (a) &
(x_{1\overline{3}},\,x_{2\overline{3}}) & (a) & (x_{1\overline{3}},\,x_{3\overline{3}}) & (a) \\[1mm]
(x_{1\overline{3}},\,x_{\widetilde{3\overline{3}}}) & (a) & (x_{2\overline{2}},\,x_{2\overline{3}}) & (a) &
(x_{\widetilde{2\overline{2}}},\,x_{2\overline{3}}) & (a) & (x_{2\overline{3}},\,x_{3\overline{3}}) & (a) \\[1mm]
(x_{2\overline{3}},\,x_{\widetilde{3\overline{3}}}) & (a) & (x_{3\overline{3}},\,x_{\widetilde{3\overline{3}}}) & (a) &
(x_{0\overline{0}},\,x_{\widetilde{0\overline{0}}}) & (b) & (x_{1\overline{1}},\,x_{\widetilde{1\overline{1}}}) & (b) \\[1mm]
(x_{2\overline{2}},\,x_{\widetilde{2\overline{2}}}) & (b) & (x_{1\overline{0}},\,x_{2\overline{0}}) & (c) &
(x_{2\overline{0}},\,x_{2\overline{1}}) & (c) & (x_{13},\,x_{1\overline{0}}) & (d) \\[1mm]
(x_{02},\,x_{2\overline{0}}) & (d) & (x_{2\overline{2}},\,x_{2\overline{1}}) & (e) &
(x_{2\overline{2}},\,x_{2\overline{0}})  & (e) & (x_{1\overline{1}},\,x_{2\overline{1}})  & (e) \\[1mm]
(x_{1\overline{1}},\,x_{2\overline{0}}) & (e) & (x_{1\overline{1}},\,x_{1\overline{0}}) & (e) &
(x_{0\overline{0}},\,x_{2\overline{0}}) & (e) & (x_{0\overline{0}},\,x_{1\overline{0}}) & (e) \\[1mm]
(x_{\widetilde{2\overline{2}}},\,x_{2\overline{1}}) & (e) & (x_{\widetilde{2\overline{2}}},\,x_{2\overline{0}}) & (e) &
(x_{\widetilde{1\overline{1}}},\,x_{2\overline{1}}) & (e) & (x_{\widetilde{1\overline{1}}},\,x_{2\overline{0}})  & (e) \\[1mm]
(x_{\widetilde{1\overline{1}}},\,x_{1\overline{0}}) & (e) & (x_{\widetilde{0\overline{0}}},\,x_{2\overline{0}}) & (e) &
(x_{\widetilde{0\overline{0}}},\,x_{1\overline{0}}) & (e) & (x_{3\overline{1}},\,x_{2\overline{1}}) & (e) \\[1mm]
\hline
\end{array} 
\]

}
\end{example}

\medskip
Now, as explained in \S\ref{estacat}, we may conclude that:
\begin{theorem} $\mathcal{C}_\xi$ is a monoidal categorification of the cluster algebra $\A$ of type $D_n$.
\end{theorem}

\bigskip
\small
\noindent
\begin{tabular}{ll}
David {\sc Hernandez}  
& Universit\'e Paris 7, Institut de Math\'ematiques de Jussieu, CNRS UMR 7586,\\
& 175 Rue du Chevaleret, F-75013 Paris, France \\
& email : {\tt hernandez@math.jussieu.fr}\\[5mm]
Bernard {\sc Leclerc}  & Universit\'e de Caen Basse-Normandie, UMR 6139 LMNO, 14032 Caen, France\\
&CNRS UMR 6139 LMNO, F-14032 Caen, France\\
&Institut Universitaire de France,\\
&email : {\tt bernard.leclerc@unicaen.fr}
\end{tabular}

\end{document}